\documentclass[a4paper,10pt]{amsart}
\usepackage{multicol}
\usepackage{amsaddr}
\usepackage{enumerate}
\usepackage{amssymb} \usepackage{amsfonts} \usepackage{amsmath}
\usepackage{amsthm}
\usepackage{mathrsfs}
\usepackage{subfig}
\usepackage{color}

\usepackage{caption}
\usepackage[utf8x]{inputenc}
\usepackage[T1]{fontenc}
\usepackage[margin=1.5in]{geometry}

\newtheorem{theorem}{Theorem}[section]
\newtheorem{lemma}[theorem]{Lemma}

\theoremstyle{definition}
\newtheorem{ex}[theorem]{Example}

\newtheorem{defn}[theorem]{Definition}

\theoremstyle{remark}

\newcommand{\p} {\ensuremath {\mathbb{P}}}
\newcommand{\E} {\ensuremath {\mathbb{E}}}

\newcommand{\R} {\ensuremath {\mathbb{R}}}

\newcommand{\I} {\ensuremath {\mathbb{I}}}

\newcommand{\A} {\ensuremath {\mathcal{A}}}

\newcommand{\D} {\ensuremath {\mathcal{D}}}

\newcommand{\No} {\ensuremath {\mathcal{N}}}

\newcommand{\tnu} {\ensuremath {\tilde{\nu}}}

\title{$L_1$-distance for additive processes with time-homogeneous Lévy measures}
\author{Pierre ~Étoré and Ester ~Mariucci}
\address{Laboratoire Jean Kuntzmann, Grenoble.}
\email{Ester.Mariucci@imag.fr}
\date{\today}
\begin{document}
\begin{abstract}
We give an explicit bound for the $L_1$-distance between two additive processes of local characteristics $(f_j(\cdot),\sigma^2(\cdot),\nu_j)$, $j = 1,2$. The cases $\sigma =0$ and $\sigma(\cdot) > 0$ are both treated. We allow $\nu_1$ and $\nu_2$ to be time-homogeneous Lévy measures, possibly with infinite variation. Some examples of possible applications are discussed.
\end{abstract}
\maketitle
\section{Introduction and main result}
In this note we give an upper bound for the $L_1$-distance between the laws induced on the Skorokhod space by two additive processes observed until time $T>0$. By the $L_1$-distance between two $\sigma$-finite measures $\mu_1$ and $\mu_2$ on $(E,\mathscr E)$ such that $\mu_1$ is absolutely continuous with respect to $\mu_2$ we mean
$$L_1(\mu_1,\mu_2)=2\sup_{A\in\mathscr E}\big|\mu_1(A)-\mu_2(A)\big|=\int_E \bigg|\frac{d\mu_1}{d\mu_2}-1\bigg|d\mu_2.$$
Note that, with our definitions, the $L_1$-distance is twice the so called total variation distance.

Giving bounds for the $L_1$-distance is a classical problem, which, in the last decades, has been reinterpreted in more modern terms via Stein's method (see, e.g., \cite{P,R,pp}). This kind of problems arises in several fields such us Bayesian statistics, convergence rates of Markov chains or Monte Carlo algorithms (see \cite{GS}, Section 4 and the references therein). However, to the best of our knowledge, results bounding the $L_1$-distance between laws on the Skorokhod space are much less abundant. In this setting other kinds of distances have been privileged such as the Wasserstein-Kantorovich-Rubinstein metric (see \cite{g13}). More relevant to our purposes is a result due to Memin and Shiryayev \cite{MS} computing the Hellinger distance between the laws of any two processes with independent increments. In particular this gives a bound for the $L_1$-distance between additive processes. In order to state their result let us fix some notation. 

Let $\{x_t\}$ be the canonical process on the Skorokhod space $(D,\D)$ and denote by $P^{(f,\sigma^2,\nu)}$ the law induced on $(D,\D)$ by an additive process having local characteristics $(f(\cdot),\sigma^2(\cdot),\nu)$. We will denote such a process by $\big(\{x_t\},P^{(f,\sigma^2,\nu)}\big)$ and we will write $P_T^{(f,\sigma^2,\nu)}$ for the restriction of $P^{(f,\sigma^2,\nu)}$ to the $\sigma$-algebra generated by $\{x_s:0\leq s\leq T\}$ (see Section \ref{ap} for the precise definitions). Our purpose is to bound $L_1\big(P_T^{(f_1,\sigma_1^2,\nu_1)},P_T^{(f_2,\sigma_2^2,\nu_2}\big)$. From now on we will assume that $\sigma_1^2(\cdot)=\sigma_2^2(\cdot)=\sigma^2(\cdot)$, otherwise this distance is $2$ (see, e.g., \cite{newman,JS}). We also need to define the following quantities:
 \begin{align*}
&\gamma^{\nu_j}=\int_{\vert y \vert \leq 1}y\nu_j(dy), \quad j=1,2;\quad 
\xi^2=\int_0^T\frac{(f_2(r)-f_1(r)-(\gamma^{\nu_2}-\gamma^{\nu_1}))^2}{\sigma^2(r)} dr.
\end{align*}
\begin{theorem}[Memin and Shiryayev]\label{teo}
 Let $\big(\{x_t\}, P^{(f_1,\sigma^2,\nu_1)}\big)$ and $\big(\{x_t\}, P^{(f_2,\sigma^2,\nu_2)}\big)$ be two additive processes with $\nu_1$ and $\nu_2$ Lévy measures such that $\nu_1$ is absolutely continuous with respect to $\nu_2$ and satisfying: 
 \begin{equation}\label{hp}
  H^2(\nu_1,\nu_2):=\int_{\R}\bigg(\sqrt{\frac{d\nu_1}{d\nu_2}(y)}-1\bigg)^2\nu_2(dy)<\infty.
 \end{equation}
The following upper bounds hold, for any $0<T<\infty$:
If $\sigma^2>0$ then
  $$L_1\Big(P_T^{(f_1,\sigma^2,\nu_1)},P_T^{(f_2,\sigma^2,\nu_2)}\Big)\leq \sqrt{8\bigg(1-\exp\Big(-\frac{\xi^2}{8}-\frac{T}{2}H^2(\nu_1,\nu_2)\Big)\bigg)}.$$
  If $\sigma^2=0$ and  $f_1-f_2\equiv\gamma^{\nu_1}-\gamma^{\nu_2}$, then
   $$L_1\Big(P_T^{(f_1,0,\nu_1)},P_T^{(f_2,0,\nu_2)}\Big)\leq \sqrt{8\bigg(1-\exp\Big(-\frac{T}{2}H^2(\nu_1,\nu_2)\Big)\bigg)}.
$$
\end{theorem}
Observe that \eqref{hp} implies $\gamma^{\nu_j}<\infty$, $j=1,2$. When $\sigma^2=0$ it follows from Theorem \ref{teo1} that, for example, $L_1\Big(P_T^{(\gamma^{\nu_1},0,\nu_1)},P_T^{(\gamma^{\nu_2},0,\nu_2)}\Big)\leq 2\sqrt{TL_1(\nu_1,\nu_2)}$.

The proof of Theorem \ref{teo}, however, makes heavy use of general theory of semimartingales. This note originated from the research for a proof based only on classical results for Lévy processes, Esscher-type transformations and the Cameron-Martin formula.
It turned out that this method, when applicable, gives sharper bound on the $L_1$-distance. More precisely, our main result is as follows.
\begin{theorem}\label{teo1}
 Let $\big(\{x_t\}, P^{(f_1,\sigma^2,\nu_1)}\big)$ and $\big(\{x_t\}, P^{(f_2,\sigma^2,\nu_2)}\big)$ be two additive processes with $\nu_1$ and $\nu_2$ Lévy measures such that $\nu_1$ is absolutely continuous with respect to $\nu_2$ and satisfying: 
 \begin{equation*}
  L_1(\nu_1,\nu_2)<\infty.
 \end{equation*}
 Then, the following upper bounds hold, for any $0<T<\infty$.
  
  If $\sigma^2>0$ then
  $$L_1\Big(P_T^{(f_1,\sigma^2,\nu_1)},P_T^{(f_2,\sigma^2,\nu_2)}\Big)\leq 2\sinh\Big(TL_1(\nu_1,\nu_2)\Big)+2\bigg[1-2\phi\Big(-\frac{\xi}{2}\Big)\bigg].$$
  If $\sigma^2=0$ and  $f_1-f_2\equiv\gamma^{\nu_1}-\gamma^{\nu_2}$, then
  
 $$L_1\Big(P_T^{(f_1,0,\nu_1)},P_T^{(f_2,0,\nu_2)}\Big)\leq 2\sinh\Big(TL_1(\nu_1,\nu_2)\Big).
$$

\end{theorem}

Remark that, in the case $\nu_1 = \nu_2 = 0$, i.e. where there are no jumps, the upper bound in Theorem \ref{teo1} is achieved. Indeed, an explicit formula for the $L_1$-distance between Gaussian processes is well known. Denoting by $\phi$ the cumulative distribution function of a normal random variable $\No(0,1)$, we have, for any $0<T<\infty$:
\begin{equation*}
L_1\big(P_T^{(f_1,\sigma^2,0)},P_T^{(f_2,\sigma^2,0)}\big)= 2\bigg(1-2\phi\bigg(-\frac{1}{2}\sqrt{\int_0^T\frac{(f_1(t)-f_2(t))^2}{\sigma_1^2(t)}dt}\bigg)\bigg)  
\end{equation*}
whenever the right-hand side term makes sense (see, e.g., \cite{BL}). 
 
The reason for our interest in the $L_1$-distance lies in the fact that it is a fundamental tool in the Le Cam theory of comparison of statistical models (\cite{lecam, LC2000}). More precisely, the presented result will be needed in a forthcoming paper by the second author, establishing an equivalence result, in the Le Cam sense, for additive processes.
Similar estimations appear in many other results concerning the Le Cam $\Delta$-distance. See for example \cite{BL,N96,cmultinomial}, where the $L_1$-distance between Gaussian processes is computed or \cite{NM,R2006, C14} concerning diffusion processes without jumps. In recent years, however, there is a growing interest in models with jumps due to their numerous applications in econometrics, insurance theory or financial modelling. Because of that, it is useful to dispose of simple formulas for estimating distances between such processes.

Theorem \ref{teo1} is proved in Section \ref{dim}. In Section \ref{ap} we collect some preliminary results about additive processes that will play a role in the proof. Before that, we give some examples of situations where our result can be applied. The choice of these examples are inspired by the models exhibited in \cite{tankov}.
\begin{ex}($L_1$-distance between compound Poisson processes)

 Let $\{X^1_t\}$ and $\{X^2_t\}$ be two compound Poisson processes on $[0,T]$ with intensities $\lambda_j>0$, $j=1,2$ and jump size distributions $G_j$; i.e. $\{X^j_t\}$ is a Lévy process of characteristic triplet $\big(\lambda_j\int_{|y|\leq 1}yG_j(dy),0,\lambda_jG_j\big)$. Furthermore, let $A$ be a subset of $\R$ and suppose that $G_j$ is equivalent to the Lebesgue measure restricted to $A$. Denote by $g_j$ the density $\frac{dG_j}{d\textnormal{Leb}_{|A}}$; then, an application of Theorem \ref{teo1} yields:
 $$L_1\big(X^1,X^2\big)\leq 2\sinh\Big(T\int_A |\lambda_1g_1(y)-\lambda_2g_2(y)|dy\Big).$$
\end{ex}
\begin{ex}($L_1$-distance between additive processes of jump-diffusion type)

 An additive process of jump-diffusion type on $[0,T]$ has the following form:
 $$X_t=\int_0^t f(r)dr+\int_0^t\sigma(r)dW_r+\sum_{i=1}^{N_t} Y_i,\quad t\in[0,T],$$
 where $\{W_t\}$ is a standard Brownian motion, $\{N_t\}$ is the Poisson process counting the jumps of $\{X_t\}$, and $Y_i$ are jumps sizes (i.i.d. random variables). Consider now the additive processes of jump-diffusion type $\{X_t^j\}$ having local characteristics $(f_j(\cdot)+\lambda_j\int_{|y|\leq 1}yG_j(dy),\sigma^2(\cdot),\lambda_jG_j)$, $j=1,2$ and suppose again that $G_j$ is equivalent to the Lebesgue measure restricted to some $A\subseteq \R$. Letting $g_j$ denote the density of $G_j$ as above, we have:
 $$L_1\big(X^1,X^2\big)\leq 2\sinh\Big(T\int_A |\lambda_1g_1(y)-\lambda_2g_2(y)|dy\Big)+2\Big(1-2\phi\Big(-\sqrt{\int_0^T\frac{(f_1(t)-f_2(t))^2}{4\sigma^2(t)}dt}\bigg)\bigg).$$
\end{ex}
\begin{ex}($L_1$-distance between tempered stable processes)

 Let $\{X_t^1\}$ and $\{X_t^2\}$ be two tempered stable processes, i.e. Lévy processes on $\R$ with no gaussian component and such that their Lévy measures $\nu_j$ have densities of the form 
 $$\frac{d\nu_j}{d\textnormal{Leb}}(y)=\frac{C_-}{|y|^{1+\alpha}}e^{-\lambda_-^j|y|}\I_{y<0}+\frac{C_+}{y^{1+\alpha}}e^{-\lambda_+^jy}\I_{y>0},\quad j=1,2,$$
 for some parameters $C_\pm>0$, $\lambda_\pm^j>0$ and $\alpha<2$. Then the hypothesis \eqref{hp} is satisfied and Theorem \ref{teo1} bounds the $L_1$-distance by:
 $$2\sinh\bigg(T\bigg[C_+\int_0^{\infty}\bigg|\frac{e^{-\lambda_+^1y}-e^{-\lambda_+^2y}}{y^{1+\alpha}}\bigg|dy+C_-\int_{-\infty}^0\bigg|\frac{e^{-\lambda_-^1|y|}-e^{-\lambda_-^2|y|}}{|y|^{1+\alpha}}\bigg|dy\bigg]\bigg).$$
\end{ex}

\section{Preliminary results}\label{ap}
\subsection{Additive processes}
\begin{defn}
A stochastic process $\{X_t\}=\{X_t:t\geq0\}$ on $\R$ defined on a probability space $(\Omega,\A,\p)$ is an \emph{additive process} if the following conditions are satisfied.
\begin{enumerate}
\item $X_0=0$ $\p$-a.s.
\item For any choice of $n\geq 1$ and $0\leq t_0<t_1<\ldots<t_n$, random variables $X_{t_0}$, $X_{t_1}-X_{t_0},\dots ,X_{t_n}-X_{t_{n-1}}$are independent.
\item There is $\Omega_0\in \A$ with $\p(\Omega_0)=1$ such that, for every $\omega\in \Omega_0$, $X_t(\omega)$ is right-continuous in $t\geq 0$ and has left limits in $t>0$.
\item It is stochastically continuous.
\end{enumerate}
\end{defn}
Thanks to the \emph{Lévy-Khintchine formula}, the characteristic function of any additive process $\{X_t\}$ can be expressed, for all $u$ in $\R$, as:
\begin{equation}\label{caratteristica}
\E\big[e^{iuX_t}\big]=\exp\Big(iu\int_0^t f(r)dr-\frac{u^2}{2}\int_0^t \sigma^2(r)dr-t\int_{\R}(1-e^{iuy}+iuy\I_{\vert y\vert \leq 1})\nu(dy)\Big), 
\end{equation}
where $f(\cdot)$, $\sigma^2(\cdot)$ are functions on $L_1[0,T]$ and $\nu$ is a measure on $\R$ satisfying
$$\nu(\{0\})=0 \textnormal{ and } \int_{\R}(|y|^2\wedge 1)\nu(dy)<\infty. $$
In the sequel we shall refer to $(f(\cdot),\sigma^2(\cdot),\nu)$ as the local characteristics of the process $\{X_t\}$ and $\nu$ will be called \emph{Lévy measure}.
This data characterises uniquely the law of the process $\{X_t\}$. In the case in which $f(\cdot)$ and $\sigma(\cdot)$ are constant functions, a process $\{X_t\}$ satisfying 
\eqref{caratteristica} is said a \emph{Lévy process} of characteristic triplet $(f,\sigma^2,\nu)$.

Let $D=D([0,\infty),\R)$ be the space of mappings $\omega$ from $[0,\infty)$ into $\R$ that are right-continuous with left limits. Define the \emph{canonical process} $x:D\to D$ by 
$$\forall \omega\in D,\quad x_t(\omega)=\omega_t,\;\;\forall t\geq 0.$$

Let $\D_t$ and $\D$ be the $\sigma$-algebras generated by $\{x_s:0\leq s\leq t\}$ and $\{x_s:0\leq s<\infty\}$, respectively (here, we use the same notations as in \cite{sato}).

Let $\{X_t\}$ be an additive process defined on $(\Omega,\A,\p)$ having local characteristics $(f(\cdot),\sigma^2(\cdot),\nu)$. It is well known that it induces a probability measure $P^{(f,\sigma^2,\nu)}$ on (D,\D) such that $\big(\{x_t\},P^{(f,\sigma^2,\nu)}\big)$ is an additive process identical in law with 
$(\{X_t\},\p)$ (that is the local characteristics of $\{x_t\}$ under $P^{(f,\sigma^2,\nu)}$ is $(f(\cdot),\sigma^2(\cdot),\nu)$). For all $t>0$ we will denote $P_t^{(f,\sigma^2,\nu)}$ for the restriction of $P^{(f,\sigma^2,\nu)}$ to $\D_t$.
In the case where $\int_{|y|\leq 1}|y|\nu(dy)<\infty$, we set $\gamma^{\nu}:=\int_{|y|\leq 1}y\nu(dy)$.
Note that, if $\nu$ is a finite Lévy measure, then the process $\big(\{x_t\},P^{(\gamma^{\nu},0,\nu)}\big)$ is a compound Poisson process.

Here and in the sequel we will denote by $\Delta x_r$ the jump of process $\{x_t\}$ at the time $r$:  
$$\Delta x_r = x_r - \lim_{s \uparrow r} x_s.$$
\begin{defn}
Consider $\big(\{x_t\},P^{(f,\sigma^2,\nu)}\big)$ and define the \emph{jump part} of $\{x_t\}$ as
\begin{equation}\label{d}
x_t^{d,\nu}=\lim_{\varepsilon\to 0}\bigg(\sum_{r\leq t}\Delta x_r \I_{\vert \Delta x_r\vert>\varepsilon}-t\int_{\varepsilon<\vert y\vert\leq 1}y\nu(dy)\bigg) \quad \textnormal{a.s.}
\end{equation}
and its \emph{continuous part} as
\begin{equation}\label{c}
x_t^{c,\nu}=x_t-x_t^{d,\nu} \quad \textnormal{a.s.}
\end{equation}
\end{defn}

We now recall the \emph{Lévy-It\^o decomposition}, i.e. the decomposition in continuous and discontinuous parts of an additive process.
\begin{theorem}[See \cite{sato}, Theorem 19.3]\label{satodec}
Consider $\big(\{x_t\},P^{(f,\sigma^2,\nu)}\big)$ and define $\{x_t^{d,\nu}\}$ and $\{x_t^{c,\nu}\}$ as in \ref{d} and \ref{c}, respectively. Then the following hold.
\begin{enumerate}[(i)]
\item There is $D_1\in \D$ with $P^{(f,\sigma^2,\nu)}(D_1)=1$ such that, for any $\omega\in D_1$, $x_t^{d,\nu}(\omega)$ is defined for all $t\in [0,T]$ and the 
convergence is uniform in $t$ on any bounded interval, $P^{(f,\sigma^2,\nu)}$-a.s.
The process $\{x_t^{d,\nu}\}$ is a Lévy process on $\R$ with characteristic triplet $(0,0,\nu)$.
\item There is $D_2\in \D$ with $P^{(f,\sigma^2,\nu)}(D_2)=1$ such that, for any $\omega\in D_2$, $x_t^{c,\nu}(\omega)$ is continuous in $t$. 
The process $\{x_t^{c,\nu}\}$ is an additive process on $\R$ with local characteristics $(f(\cdot),\sigma^2(\cdot),0)$.
\item The two processes $\{x_t^{d,\nu}\}$ and $\{x_t^{c,\nu}\}$ are independent.
\end{enumerate}
\end{theorem}
\subsection{Change of measure for additive processes}
For the proof of Theorem \ref{teo1} we also need some results on the equivalence of measures for additive processes. 
By the notation $\ll$ we will mean ``is absolutely continuous with respect to''.

\subsubsection{Case $\sigma^2=0$}

\begin{theorem}[See \cite{sato}, Theorems 33.1--33.2 and \cite{sato2} Corollary 3.18, Remark 3.19]\label{teosato}
 Let $\big(\{x_t\},P^{(0,0,\tnu)}\big)$ and $\big(\{x_t\},P^{(\eta,0,\nu)}\big)$ be two Lévy processes on $\R$, where
\begin{equation}\label{gamma*}
 \eta=\int_{\vert y \vert \leq 1}y(\nu-\tnu)(dy)
\end{equation}
is supposed to be finite. Then $P_t^{(\eta,0,\nu)}\ll P_t^{(0,0,\tnu)}$ for all $t\geq 0$ if and only if $\nu\ll\tnu$ and the density $\frac{d\nu}{d\tnu}$ satisfies
\begin{equation}\label{Sato}
 \int\bigg(\sqrt{\frac{d\nu}{d\tnu}(y)}-1\bigg)^2\tnu(dy)<\infty.
\end{equation}
Remark that the finiteness in \eqref{Sato} implies that in \eqref{gamma*}. When $P_t^{(\eta,0,\nu)}\ll P_t^{(0,0,\tnu)}$, the density is
$$\frac{dP_t^{(\eta,0,\nu)}}{dP_t^{(0,0,\tnu)}}(x)=\exp(U_t(x)),$$
with
\begin{equation}\label{U}
 U_t(x)=\lim_{\varepsilon\to 0} \bigg(\sum_{r\leq t}\ln \frac{d\nu}{d\tnu}(\Delta x_r)\I_{\vert\Delta x_r\vert>\varepsilon}-
\int_{\vert y\vert > \varepsilon} t\bigg(\frac{d\nu}{d\tnu}(y)-1\bigg)\tnu(dy)\bigg),\\ P^{(0,0,\tnu)}\textnormal{-a.s.}
\end{equation}
The convergence in \eqref{U} is uniform in $t$ on any bounded interval, $P^{(0,0,\tnu)}$-a.s.
Besides, $\{U_t(x)\}$ defined by \eqref{U} is a Lévy process satisfying $\E_{P^{(0,0,\tnu)}}[e^{U_t(x)}]=1$, $\forall t\in [0,T]$.
\end{theorem}
\subsubsection{Case $\sigma^2>0$}
\begin{lemma}\label{lemma1}
 Let $\nu_1\ll\nu_2$ be Lévy measures such that 
\begin{align}
\int_{\R} \bigg(\sqrt{\frac{d\nu_1}{d\nu_2}(y)}-1\bigg)^2 \nu_2(dy) <\infty. \label{Sato*}
\end{align}
Define
\begin{align}\label{ngamma}
 \eta=\int_{\vert y \vert \leq 1}y(\nu_1-\nu_2)(dy),
\end{align}
which is finite thanks to \eqref{Sato*}, and consider real functions $f_1$, $f_2$ and $\sigma>0$ such that 
\begin{equation}\label{eq:eta}
 \int_0^T \Big(\frac{f_1(r)-f_2(r)-\eta}{\sigma(r)}\Big)^2dr<\infty,\quad T\geq 0.
\end{equation}
Then, under $P^{(f_2,\sigma ^2,\nu_2)}$,
\begin{align}\label{mart}
 M_t(x)&=\exp\big(C_t(x)+ D_t(x)\big)
\end{align}
is a $(\D_t)$-martingale for all $t$ in $[0,T]$, where
$$C_t(x):=\int_0^t \frac{f_1(r)-f_2(r)-\eta}{\sigma^2(r)}(dx_r^{c,\nu_2}-f_2(r)dr)-\frac{1}{2}\int_0^t \Big(\frac{f_1(r)-f_2(r)-\eta}{\sigma(r)}\Big)^2 dr,$$
\begin{equation}\label{ut}
D_t(x):=\lim_{\varepsilon\to 0}\bigg(\sum_{r\leq t}\ln \frac{d\nu_1}{d\nu_2}(\Delta x_r)\I_{\vert\Delta x_r\vert>\varepsilon}-t\int_{\vert y\vert > \varepsilon} (\nu_1-\nu_2)(dy)\bigg).
\end{equation}
The convergence in \eqref{ut} is uniform in $t$ on any bounded interval, $P^{(f_2,\sigma ^2,\nu_2)}$-a.s.
\end{lemma}
\begin{proof}
The existence of the limit in \eqref{ut} is guaranteed by \eqref{Sato*} (see Theorem \ref{teosato}). 
Since $\int_0^t \frac{1}{\sigma(r)}(dx_r^{c,\nu_2}-f_2(r)dr)$ is a standard Brownian motion under $P^{(f_2,\sigma ^2,0)}$, we have that $\int_s^t \frac{f_1(r)-f_2(r)-\eta}{\sigma^2(r)}(dx_r^{c,\nu_2}-f_2(r)dr)$ has normal law 
$\No\Big(0,\int_s^t \big(\frac{f_1(r)-f_2(r)-\eta}{\sigma(r)}\big)^2 dr\Big)$, hence $\E_{P^{(f_2,\sigma ^2,0)}}[\exp((C_t-C_s)(x))]=1$. 
Theorem \ref{satodec} entails that $\{x_t^{c,\nu_2}\}$ and $\{x_t^{d,\nu_2}\}$ are independent under $P^{(f_2,\sigma^2,\nu_2)}$. Moreover, the law of $\{C_t(x)\}$ (resp. $\{D_t(x)\}$) is the same under $P^{(f_2,\sigma^2,\nu_2)}$ or $P^{(f_2,\sigma^2,0)}$ (resp. $P^{(f_2,0,\nu_2)}$ or $P^{(0,0,\nu_2)}$). Further, using Theorem \ref{teosato}, we know that $\{D_t(x)\}$ is a Lévy process such that $\E_{P^{(0,0,\nu_2)}}[\exp(D_{t-s}(x))]=1$ for all $s<t$.
These facts together with the independence of the increments of $(\{x_t\},P^{(f_2,\sigma^2,\nu_2)})$ and the stationarity of $\{D_t(x)\}$ imply: 
\begin{align*}
\E_{P^{(f_2,\sigma^2,\nu_2)}}[M_t(x)|\D_s]&=\E_{P^{(f_2,\sigma^2,\nu_2)}}\Big[M_s(x)\exp\Big((C_t-C_s)(x)+(D_t-D_s)(x)\Big)\big|\D_s\Big]\\
            &=M_s(x)\E_{P^{(f_2,\sigma^2,\nu_2)}}[\exp((C_t-C_s)(x)+(D_t-D_s)(x))]\\
            &=M_s(x)\E_{P^{(f_2,\sigma^2,0)}}[\exp((C_t-C_s)(x))]\E_{P^{(0,0,\nu_2)}}[\exp((D_t-D_s)(x))]\\
            &=M_s(x)\E_{P^{(0,0,\nu_2)}}[\exp(D_{t-s}(x))]\\
            &=M_s(x).
\end{align*}
\end{proof}
 \begin{lemma}\label{lemmadensita}
Suppose that the hypothesis \eqref{Sato*} and \eqref{eq:eta} of Lemma \ref{lemma1} are satisfied. Then, using the same notations as above,
$P_t^{(f_1,\sigma ^2,\nu_1)}\ll P_t^{(f_2,\sigma ^2,\nu_2)}$ for all $t$ and the density is given by:
\begin{equation}\label{eq:den}
 \frac{dP_t^{(f_1,\sigma ^2,\nu_1)}}{dP_t^{(f_2,\sigma ^2,\nu_2)}}(x)=M_t(x).
\end{equation}
\end{lemma}
\begin{proof}
For $s<t$, we prove that $\E_{P^{(f_2,\sigma^2,\nu_2)}}\big[\exp(iu(x_t-x_s))\frac{M_t}{M_s}(x)\vert \D_s\big]=\E_P^{(f_1,\sigma^2,\nu_1)}[\exp(iu(x_t-x_s)]$. 
To that aim remark that, thanks again to Theorem \ref{satodec}:
\begin{align}
\E_{P^{(f_2,\sigma^2,\nu_2)}}\bigg[e^{iu(x_t-x_s)}\frac{M_t(x)}{M_s(x)}\Big|\D_s\bigg]&=\E_{P^{(f_2,\sigma^2,\nu_2)}}\bigg[e^{iu(x^{c,\nu_2}_t-x_s^{c,\nu_2}+x_t^{d,\nu_2}-x_s^{d,\nu_2})}\frac{M_t(x)}{M_s(x)}\Big|\D_s\bigg]\nonumber\\
&=\E_{P^{(f_2,\sigma^2,0)}}\Big[e^{iu(x_t-x_s)}e^{(C_t-C_s)(x)}\Big]\E_{P^{(0,0,\nu_2)}}\Big[e^{iu(x_t-x_s)}e^{(D_t-D_s)(x)}\Big].\label{3}
\end{align}
Let us now compute the first factor of \eqref{3}:
\begin{align*}
 \E_{P^{(f_2,\sigma^2,0)}}\Big[e^{iu(x_t-x_s)}e^{(C_t-C_s)(x)}\Big]&=\E_{P^{(f_1-\cdot\eta,\sigma^2,0)}}\Big[e^{iu(x_t-x_s)}\Big]\\
&=\exp\Big(iu\int_s^t(f_1(r)-\eta)dr-\frac{u^2}{2}\int_s^t\sigma^2(r)dr\Big)\nonumber.
\end{align*}
In the first equality we used the Girsanov theorem, thanks to the fact that $\int_0^t \frac{1}{\sigma(r)}(dx_r-f_2(r)dr)$ is a Brownian motion under $P^{(f_2,\sigma^2,0)}$, while the second one follows from \eqref{caratteristica}. We compute the second factor of \eqref{3} by means of Theorem \ref{teosato} and another application of \eqref{caratteristica}:
\begin{align*}
 \E_{P^{(0,0,\nu_2)}}\Big[e^{iu(x_t-x_s)}e^{(D_t-D_s)(x)}\Big]&=\E_{P^{(0,0,\nu_2)}}\Big[e^{iux_{t-s}}e^{D_{t-s}(x)}\Big]\\
                                            &=\E_{P^{(\eta,0,\nu_1)}}\Big[e^{iux_{t-s}}\Big]\\
                                            &=\exp\Big((t-s)\Big[iu\eta-\int_{\R}(1-e^{iuy}+iuy\I_{|y|\leq 1})\nu_1(dy)\Big]\Big).
\end{align*}
Consequently:
\begin{equation}\label{aux}
\E_{P^{(f_2,\sigma^2,\nu_2)}}\bigg[e^{iu(x_t-x_s)}\frac{M_t(x)}{M_s(x)}\Big|\D_s\bigg]=\E_{P^{(f_1,\sigma^2,\nu_1)}}[e^{iu(x_t-x_s)}]\quad \forall 0\leq s\leq t. 
\end{equation}
Fix $t$ and define a probability measure $P_t$ on $\D_t$ by $P_t(B)=\E_{P^{(f_2,\sigma^2,\nu_2)}}[M_t\I_B]$ for $B\in \D_t$. As a consequence of Lemma \ref{lemma1} and the Bayes rule, 
the two processes given by $\big(\{x_s: 0\leq s \leq t\},P_t^{(f_1,\sigma^2,\nu_1)}\big)$ and $\big(\{x_s: 0\leq s \leq t\},P_t\big)$
are identical. Indeed, by \eqref{aux}, both have independent increments and the prescribed characteristic function.
Consequently, \eqref{eq:den} holds. 
 \end{proof}
 \section{Proof of Theorem \ref{teo1}}\label{dim}
For the proof we will need the following three calculus lemmas.
\begin{lemma}\label{L1facile}
 Let $X$ be a random variable with normal law $\No(m,\sigma^2)$. Then
$$\E\Big| 1-e^X \Big|=2\Big[\phi\Big(-\frac{m}{\sigma}\Big)-\phi\Big(-\frac{m}{\sigma}-\sigma\Big)\Big],$$
where $\phi(x)=\frac{1}{\sqrt{2\pi}}\int\limits_{-\infty}^x e^{-\frac{y^2}{2}}dy$.
\end{lemma}
\begin{proof}
By definition we have
 \begin{align*}
  \E\Big|1-e^X\Big|&=\frac{1}{\sqrt{2\pi}\sigma}\int_{-\infty}^{\infty} \vert1-e^x \vert e^{-\frac{(x-m)^2}{2\sigma^2}}dx\\
                  &=\frac{1}{\sqrt{2\pi}\sigma}\bigg(\int_{-\infty}^{0} (1-e^x)  e^{-\frac{(x-m)^2}{2\sigma^2}}dx+\int_0^{\infty}(e^x-1)  e^{-\frac{(x-m)^2}{2\sigma^2}}dx\bigg).
 \end{align*}
 To conclude, just split the sums inside the integrals and use the change of variables $\big(y=\frac{x-m}{\sigma}-\sigma\big)$, resp. $\big(y=\frac{x-m}{\sigma}\big)$.
\end{proof}
\begin{lemma}\label{dis}
 For all $x,y$ in $\R$ we have:
\begin{equation}\label{eq:dis}
\vert 1-e^{x+y} \vert \leq \frac{1+e^x}{2}\vert 1-e^y\vert+\frac{1+e^y}{2}\vert 1-e^x\vert.
\end{equation}
\end{lemma}
\begin{proof}
By symmetry we restrict to $x\geq 0$.
\begin{itemize}
 \item $x,y\geq 0$: In this case we have that $\vert 1-e^{x+y} \vert$ is exactly equal to $\frac{1+e^x}{2}\vert 1-e^y\vert+\frac{1+e^y}{2}\vert 1-e^x\vert$.
\item $x\geq0,y\leq 0,x+y\geq 0$: Then the member on the right hand side of \eqref{eq:dis} is equal to $e^x-e^y\geq e^x-1\geq e^{x+y}-1$.
\item $x\geq0,y\leq 0,x+y\leq 0$: In this case the member on the right of \eqref{eq:dis} is equal to $e^x-e^y\geq 1-e^y\geq 1-e^{x+y}$.
\end{itemize}
\end{proof}
\begin{lemma}\label{lemmadisc}
With the same notations as in Theorem \ref{teo1} and Lemma \ref{lemma1}, we have:
 \begin{equation}\label{eq:ipd}
 \E_{P_T^{(0,0,\nu_2)}}\Big[\big| 1-\exp(D_T(x))\big|\Big]= \E_{P_T^{(\gamma^{\nu_2},0,\nu_2)}}\Big[\big| 1-\exp(D_T(x))\big|\Big]\leq 2\sinh\bigg(T\int_{\R}L_1(\nu_1,\nu_2)\bigg).
 \end{equation}
\end{lemma}
\begin{proof}
Because of Theorem \ref{satodec} it is clear that $\E_{P_T^{(0,0,\nu_2)}}\big[| 1-\exp(D_T(x))|\big]= \E_{P_T^{(\gamma^{\nu_2},0,\nu_2)}}\big[| 1-\exp(D_T(x))|\big].$
 In order to simplify the notations let us write 
 $$A^{\pm}(x):=\lim_{\varepsilon\to 0}\bigg(\sum_{r\leq T}\ln h^{\pm}(\Delta x_r)\I_{|\Delta(x_r)|>\varepsilon}-
T\int_{|y|>\varepsilon}(h^{\mp}(y)-1)\nu_2(dy)\bigg)$$ 
with $h^+= \Big(\frac{d\nu_1}{d\nu_2}\Big)^{\I_{\frac{d\nu_1}{d\nu_2}\geq 1}}$ and $h^-=\Big( \frac{d\nu_1}{d\nu_2}\Big)^{\I_{\frac{d\nu_1}{d\nu_2}<1}}$, so that
$$D_T(x)=A^+(x)+A^-(x).$$
Then, using Lemma \ref{dis} and the fact that $A^+(x)\geq 0$ and $A^-(x)\leq 0$ we get:
\begin{align*}
 \E_{P_T^{(\gamma^{\nu_2},0,\nu_2)}}\big[&\vert 1-D_T(x) \vert\big]=\E_{P_T^{(\gamma^{\nu_2},0,\nu_2)}}\big| 1-\exp(A^+(x)+A^-(x))\big| \nonumber\\
                              &\leq \E_{P_T^{(\gamma^{\nu_2},0,\nu_2)}}\bigg[\frac{1+e^{A^+(x)}}{2}\Big| 1-e^{A^-(x)}\Big|+\frac{1+e^{A^-(x)}}{2}\Big| 1-e^{A^+(x)}\Big|\bigg] \label{eq:pass1}\\
                              &=\E_{P_T^{(\gamma^{\nu_2},0,\nu_2)}}\Big[e^{A^+(x)}-e^{A^-(x)}\Big].
\end{align*}
In order to compute the last quantity we apply Theorem \ref{teosato} and the fact that both $A^+(x)$ and $A^-(x)$ have the same law under $P_T^{(\gamma^{\nu_2},0,\nu_2)}$ and 
$P_T^{(0,0,\nu_2)}$:
\begin{align*}
   \E_{P_T^{(\gamma^{\nu_2},0,\nu_2)}}\Big[e^{A^+(x)}-e^{A^-(x)}\Big] &=\exp\bigg(T\int_{\R}(h^+(y)-h^-(y))\nu_2(dy)\bigg)\nonumber\\
&\phantom{=}-\exp\bigg(T\int_{\R}(h^-(y)-h^+(y))\nu_2(dy)\bigg)\\
      &=2\sinh\bigg(T\int_{\R}(h^+(y)-h^-(y)\nu_2(dy)\bigg)\nonumber \\
&=2\sinh\bigg(T\int_{\R}\Big| 1-\frac{d\nu_1}{d\nu_2}(y)\Big| \nu_2(dy)\bigg)\nonumber.
\end{align*}
\end{proof}
\begin{proof}[Proof of Theorem \ref{teo1}]
 
 \emph{Case $\sigma^2>0$:} With the same notations as in Lemma \ref{lemma1} and by means of Lemma \ref{lemmadensita} one can write
\begin{align*}
L_1\big(P_T^{(f_1,\sigma^2,\nu_1)},P_T^{(f_2,\sigma^2,\nu_2)}\big)&=\E_{P_T^{(f_2,\sigma^2,\nu_2)}}\big|1-\exp(C_T(x)+D_T(x))\big|.
\end{align*}
Now, using Lemma \ref{dis} and the independence between 
$C_T(x)$ and $D_T(x)$ (Theorem \ref{satodec}), we obtain
\begin{align*}
  L_1\big(P_T^{(f_2,\sigma^2,\nu_2)},P_T^{(f_1,\sigma^2,\nu_1)}\big)\leq& \E_{P_T^{(f_2,\sigma^2,\nu_2)}}\bigg(\frac{1+e^{C_T(x)}}{2}\bigg)\E_{P_T^{(f_2,\sigma^2,\nu_2)}} \vert 1-e^{D_T(x)} \vert\\
   &+\E_{P_T^{(f_2,\sigma^2,\nu_2)}}\bigg(\frac{1+e^{D_T(x)}}{2}\bigg)\E_{P_T^{(f_2,\sigma^2,\nu_2)}}\vert 1-e^{C_T(x)}\vert.
   \end{align*}
   We conclude the proof using Lemmas \ref{lemmadisc} and \ref{L1facile} together with the fact that $\E_{P_T^{(f_2,\sigma^2,\nu_2)}} e^{C_T(x)}=1=\E_{P_T^{(f_2,\sigma^2,\nu_2)}}e^{D_T(x)}$.
   
\emph{Case $\sigma^2=0$:} If $f_1-f_2\equiv\gamma^{\nu_1}-\gamma^{\nu_2}$, notice that, as the drift component of $\big(\{x_t\},P_T^{(f_1,0,\nu_1)}\big)$ and $\big(\{x_t\},P_T^{(f_2,0,\nu_2)}\big)$ is deterministic, we have $$\frac{dP_T^{(f_1,0,\nu_1)}}{dP_T^{(f_2,0,\nu_2)}}(x)=\frac{dP_T^{(f_1-f_2,0,\nu_1)}}{dP_T^{(0,0,\nu_2)}}(x)=D_T(x)$$ with $D_T(x)$ as in \eqref{ut}. Theorem \ref{teosato} allows us to write the $L_1$-distance between $P_T^{(f_1,0,\nu_1)}$ and $P_T^{(f_2,0,\nu_2)}$ as $\E_{P_T^{(f_2,0,\nu_2)}}\big| 1-D_T(x)\big|$. We then obtain the bound $2\sinh(TL_1(\nu_1,\nu_2))$ by means of Lemma \ref{lemmadisc}.
\end{proof}

\bibliographystyle{plain}
\bibliography{refs}

%
 

\end{document}